\documentclass{amsart}
\usepackage{amssymb,latexsym,graphics,amsfonts}

\swapnumbers
 \newtheorem{theo}{Theorem}[section]
\newtheorem{cor}[theo]{Corollary}
\newtheorem{lemm}[theo]{Lemma}
\newtheorem{prop}[theo]{Proposition}
\theoremstyle{definition}
\newtheorem{defi}[theo]{Definition}
\theoremstyle{definition}
\newtheorem{ex}[theo]{Example}

\numberwithin{equation}{section}

 \newcommand{\A}{\mathcal{A}}
  
 \newcommand{\J}{\mathcal{J}}

\begin{document}

\title{Module Pseudo-amenability of Banach algebras}

\author[A. Bodaghi]{Abasalt Bodaghi}
\address{Department of Mathematics, Garmsar Branch, Islamic Azad University,  Garmsar,
IRAN}

 \email{abasalt.bodaghi@gmail.com}
\author[A. Jabbari]{Ali Jabbari}
\address{Young Researchers and Elite Club, Ardabil Branch, Islamic Azad University, Ardabil, IRAN}
\email{jabbari\_al@yahoo.com}

\subjclass[2010]{Primary 46H20; Secondary 46H25, 16E40.}

\keywords{ Approximately inner; Module approximate amenability; Module derivation;  Module pseudo-amenability}

\dedicatory{}

\smallskip

\begin{abstract}
The notions of module pseudo-amenable and module pseudo-contractible Banach algebras are introduced. For a Banach algebra with bounded approximate identity, module pseudo-amenability and module approximate amenability are the same properties. It is given a complete characterization of module pseudo-amenability for a Banach algebra. For every inverse semigroup $S$ with subsemigroup $E$ of idempotents, necessary and sufficient conditions are obtained for the  $\ell^1(S)$ and its second dual to be $\ell^1(E)$-module pseudo-amenable.

\end{abstract}

 \maketitle


\section{introduction}
Ghahramani and Zhang in \cite{gz} introduced two new notions of Johnson's amenablility without boundedness; pseudo-amenability and pseudo-contractibility. They compared these notions with some of those that were investigated earlier in \cite{ga1} and \cite{ga2}. They also studied pseudo-amenability and pseudo-contractibility of Banach algebras associated to locally compact groups such as group algebras, measure
algebras and Segal algebras. Indeed, for a locally ompat group $G$, the group algebra $L^1(G)$ is pseudo-amenable if and only if $G$ is amenable.

The concept of module amenability for a class of Banach
algebras which is in fact a generalization of the amenability has been developed by Amini in \cite{am1}. He showed that for every inverse semigroup $S$ with subsemigroup $E$ of idempotents, the $\ell^1(S)$-module amenability of $\ell^1(S)$ is equivalent to the amenability
of $S$. This concept was modified in \cite{boda1} and \cite{boda2}, by using module homomorphisms between Banach algebras. In \cite{bp}, Pourmahmood and Bodaghi introduced the notions of module approximate amenability and contractibility of Banach algebras that are modules over another Banach algebra. They proved that $\ell^1(S)$ is module approximately amenable (contractible) if and only if $S$ is amenable (for the module character amenability and the permanent weak module amenability of inverse semigroup algebras refer to \cite{bod2} and \cite{bod3}, respectively).

In the current work, we define the concepts of module pseudo-amenability and module pseudo-contractibility for Banach algebras. We characterize module approximately amenable in different way with \cite{bp}; through vanishing of the first module cohomology groups $H^1_{\mathfrak
A}(\mathcal A, X^{**})$ for certain Banach $\mathcal A$-$\mathfrak A$-bimodules $X$. We also show that module pseudo-amenability of a Banach algebra is equivalent to existence of a module approximate morphism. It is proved that a module pseudo-amenable Banach
algebra is in fact module approximately  amenable under a condition. It is shown that module pseudo-amenability of $\mathcal A^{**}$ implies module pseudo-amenability of $\mathcal A$. As consequences, we deduce that when $\ell^{1}(E)$ acts on
$\ell^{1}(E)$ trivially from left and by multiplication from
right, the semigroup algebra $ \ell ^{1}(S)$ is
$\ell^{1}(E)$-module pseudo-amenable if and only if $S$ is amenable. Also, $\ell ^1(S)^{**}$ is module pseudo-amenable (with respect to
the above actions)  if and only if the maximal group homomorphic image
$G_S=S/\approx$ is finite, where $s\approx t$ whenever
$\delta_s-\delta_t$ belongs to the closed linear span of the set
$\{\delta_{set}-\delta_{st}: s,t\in S, e\in E\}.$



\section{Notation and preliminary results}

Let $\mathcal A$ be a Banach algebra, and let $X$ be a Banach $\mathcal A$-bimodule. A bounded (continuous)
linear map $D: \mathcal A \longrightarrow X$ is called a {\it derivation} if
$$D(ab)=D(a)\cdot b+a\cdot D(b) \qquad (a,b \in \mathcal A).$$

For each $x\in X$, the derivation ${\rm{ad}}_x: \mathcal A \longrightarrow X$ defined by
${\rm{ad}}_x(a)=a\cdot x-x\cdot a$ is called {\it inner}. A derivation $D: \mathcal A
\longrightarrow X$ is said to be {\it approximately inner} if there exists a net $(x_i)\subseteq X$ such that
$D(a)=\lim_i(a\cdot x_i-x_i\cdot a)$ for all $a \in \mathcal A.$
Hence $D$ is approximately inner if it is in the closure of the
set of inner derivations with respect to the strong operator
topology on $\bf B(\mathcal A)$, the space of bounded linear operators on $\mathcal A$. The Banach algebra $\mathcal A$ is {\it approximately amenable} ({\it contractible}) if every bounded derivation $D: \mathcal A \longrightarrow X^*$ ($D: \mathcal A \longrightarrow X$) is
approximately inner, for each Banach $\mathcal A$-bimodule $X$ \cite{ga1}, where $X^*$ denotes the first dual space of $X$ which is a Banach $\mathcal A$-bimodule in the canonical way. We use the qualifier $w^*$ when that convergence is in the appropriate
weak$^*$-topology.

Let ${\mathcal A}$ and ${\mathfrak A}$ be Banach algebras such
that ${\mathcal A}$ is a Banach ${\mathfrak A}$-bimodule with
compatible actions as follows:
$$ \alpha\cdot(ab)=(\alpha\cdot a)b,
\quad(ab)\cdot\alpha=a(b\cdot\alpha) \qquad(a,b \in {\mathcal
A},\alpha\in {\mathfrak A}).$$

Let ${X}$ be a left Banach ${\mathcal A}$-module and a
Banach ${\mathfrak A}$-bimodule with the following compatible actions:
$$\alpha\cdot(a\cdot x)=(\alpha\cdot a)\cdot x,
\,\,a\cdot(\alpha\cdot x)=(a\cdot\alpha)\cdot x, \,\, a\cdot
(x\cdot\alpha)=(a\cdot x)\cdot \alpha \qquad (a \in{\mathcal
A},\alpha\in {\mathfrak A},x\in{X}).$$

Then we say that ${X}$ is a {\it{left Banach ${\mathcal A}$-${\mathfrak A}$-module}}.
Right Banach $\mathcal A$-$\mathfrak A$-modules and (two-sided) Banach $\mathcal A$-$\mathfrak A$-modules are defined similarly.
If $\alpha\cdot x=x\cdot\alpha $ for all $\alpha\in
{\mathfrak A}$ and $x\in{ X} $, then $ X $ is called a {\it
commutative} left (right or two-sided) Banach ${\mathcal A}$-${\mathfrak A}$-module. If $X $ is a
(commutative) Banach ${\mathcal A}$-${\mathfrak A}$-module, then so
is $X^*$, where the actions of $\mathcal A$ and ${\mathfrak A}$
on $ X^*$ are defined as usual \cite{am1}.

 Note that in general, ${\mathcal A}$ is not an ${\mathcal A}$-${\mathfrak
A}$-module because ${\mathcal A}$ does not satisfy the
compatibility condition $a\cdot(\alpha\cdot
b)=(a\cdot\alpha)\cdot b$ for $\alpha\in {\mathfrak A}, a,b
\in{\mathcal A}$. But when $\mathcal A$ is a commutative
$\mathfrak A$-module and acts on itself by multiplication from
both sides, then it is also a Banach ${\mathcal A}$-${\mathfrak
A}$-module.

Let ${\mathcal A}$ and ${\mathfrak A}$ be as above
and $X$ be a Banach ${\mathcal A}$-$\mathfrak A$-module. A ($\mathfrak A$-){\it
module derivation} is a bounded (continuous) map $D: \mathcal A \longrightarrow X $
satisfying
$$D(a\pm b)=D(a)\pm D(b), \quad D(ab)=D(a)\cdot b+a\cdot D(b) \qquad
(a,b \in \mathcal A),$$ and
$$D(\alpha\cdot a)=\alpha\cdot D(a), \quad D(a\cdot\alpha)=D(a)\cdot\alpha
\qquad (a \in{\mathcal A},\alpha\in {\mathfrak A}).$$

 Note that $D: \mathcal A \longrightarrow X $ is bounded if
there exists $M>0$ such that $\|D(a)\|\leq M\|a\| $, for
each $a \in\mathcal A$.  When $X $ is commutative $\mathfrak A$-module, each
$x \in X $ defines a module derivation which is inner as follows:
$$D_x(a)=a\cdot x-x\cdot a \qquad (a \in{\mathcal A}).$$

Consider the module projective tensor product ${\mathcal
A}\widehat{\otimes} _{\mathfrak A} {\mathcal A}$ which is isomorphic
to the quotient space $(\mathcal A \widehat{\otimes} \mathcal
A)/{\mathcal I_{\mathcal A}}$, where $\mathcal I_{\mathcal A}$ is the closed linear span of $\{ a\cdot\alpha \otimes b-a
\otimes\alpha \cdot b : \alpha\in {\mathfrak A},a,b\in{\mathcal A}\}$. Also consider the closed ideal
$\mathcal J_{\mathcal A}$ of ${\mathcal A}$ generated by elements of the form $
(a\cdot\alpha) b-a(\alpha \cdot b)$ for $ \alpha\in {\mathfrak
A},a,b\in{\mathcal A}$. We shall denote $\mathcal I_{\mathcal A}$ and
$\mathcal J_{\mathcal A}$ by $\mathcal I$ and $\mathcal J$, respectively, if there is no risk of confusion. Then $\mathcal I$ and $\mathcal J$ are ${\mathcal A}$-submodules and ${\mathfrak A}$-submodules of
$\mathcal A \widehat{\otimes} \mathcal A$ and $\mathcal A$,
respectively, and the quotients $\mathcal A
\widehat{\otimes}_{\mathfrak A} \mathcal A$ and $\mathcal A/\mathcal J$ are
${\mathcal A}$-modules and ${\mathfrak A}$-modules. Also,
$\mathcal A/\mathcal J$ is a Banach $\mathcal A$-${\mathfrak A}$-module
 when ${\mathcal A}$ acts on ${\mathcal
A}/\mathcal J$ canonically. Also, let $\omega_{\mathcal{A}}:
\mathcal{A}\widehat{\otimes}\mathcal{A}\longrightarrow \mathcal{A}$ be the product map, i.e., $\omega_{\mathcal{A}}(a\otimes b)=ab$, and let
$\widetilde{\omega}_{\mathcal{A}} : \mathcal{A}\widehat{\otimes}_{\mathfrak{A}}\mathcal{A} =
{(\mathcal{A}\widehat{\otimes}\mathcal{A}})/{{\mathcal I}}\longrightarrow
{\mathcal{A}}/{{\mathcal J}}$ be its induced product map, i.e.,
$\widetilde{\omega}_{\mathcal{A}}(a\otimes
b+{\mathcal I})=ab+{\mathcal J}$.

Recall that a Banach algebra $\mathcal A$ is said to be {\it module approximately amenable} ({\it contractible}) if for any commutative Banach $\A$-$\mathfrak A$-module $X$, every module derivation $D:\A\longrightarrow X^*$ ($D:\A\longrightarrow X$) is approximately inner.

\begin{defi}\label{def2}  A Banach algebra $\mathcal A$ is said to be {\it module pseudo-amenable} if there is a net $\{\widetilde{u}_j \}$ in ${\mathcal A}\widehat
\otimes _{\mathfrak A} {\mathcal A}$ which is called a {\it module
approximate diagonal} for $\mathcal A$, so that $\widetilde{\omega}_{\mathcal A}(
\widetilde{u}_j) $ is an approximate identity of $
\mathcal A/\mathcal J$ and
 $\widetilde{u}_j \cdot a- a\cdot \widetilde{u}_j \rightarrow 0$ for all $a \in \mathcal A$
\end{defi}

\begin{defi}\label{def3}  A Banach algebra $\mathcal A$ is called {\it module pseudo-contractible} if there is a net $\{\widetilde{u}_j \}$ in ${\mathcal A}\widehat
\otimes _{\mathfrak A} {\mathcal A}$, say {\it module central
approximate diagonal} for $\mathcal A$, such that
$ \widetilde{u}_j \cdot a= a\cdot \widetilde{u}_j$ for all $a \in \mathcal A$
and $\widetilde{u}_j$.
\end{defi}


\section{Main results}

Let $\mathfrak A$ be a non-unital Banach algebra. Then $\mathfrak A^\sharp=\mathfrak A\oplus \mathbb{C}$, the unitization of $\mathfrak A$, is a unital
Banach algebra which contains $\mathfrak A$ as a closed ideal. Let $\mathcal A$ be a Banach algebra and a Banach
$\mathfrak A$-bimodule with compatible actions. Then $\mathcal A$ is a Banach algebra and a Banach $\mathfrak A^\sharp$-bimodule
with compatible actions in the obvious way, i.e., the action of $\mathfrak A^\sharp$ on $\A$ is as follows:
$$(\alpha, \lambda)\cdot a =\alpha\cdot a + \lambda a,~ a\cdot (\alpha, \lambda) = a\cdot\alpha + \lambda a \hspace{1cm}(\lambda\in\mathbb{C}, \alpha\in\mathfrak A, a\in\A).$$

Let $\A$ be a Banach algebra and a Banach $\mathfrak A$-bimodule with compatible actions and let $\mathcal B =(\mathcal A \oplus\mathfrak A^\sharp,\bullet)$, where the multiplication $``\bullet"$ is defined through
$$(a, u) \bullet (b, v) = (ab + av + ub, uv)\hspace{1cm} (a, b\in\mathcal A, u, v\in\mathfrak A^\sharp).$$

Then with the following actions
$$u \cdot(a, v) = (u \cdot a, uv), ~(a, v) \cdot u = (a \cdot u, vu) \hspace{1cm}(a\in\mathcal A, u, v\in\mathfrak A^\sharp),$$
 $\mathcal B$ is a unital Banach algebra and a Banach $\mathfrak A^\sharp$-bimodule with compatible actions. The proof of the following lemma is similar to Lemma 2.2 in \cite{ga1}, so we omit it.

\begin{lemm}\label{l1}
Let $\mathcal A$  be a commutative Banach ${\mathfrak A}$-module. If $\mathcal A$ is module approximately amenable, then it has left and right approximate identity.
\end{lemm}

\begin{theo}\label{t1}
Let $\mathcal A$  be a Banach ${\mathfrak A}$-module with bounded approximate identity. Then $\mathcal A$ is module pseudo-amenable if and only if $\mathcal A$ is module approximately  amenable.
\end{theo}
\begin{proof}
Since $\mathcal A$ is module pseudo-amenable, there is a  module approximate diagonal $\{\widetilde{u}_j \}$ in ${\mathcal A}\widehat
\otimes _{\mathfrak A} {\mathcal A}$ such that
\begin{equation*}
    \widetilde{u}_j \cdot a- a\cdot \widetilde{u}_j \rightarrow 0,~\emph{\emph{and}}~(a+\J)\widetilde{\omega}_{\mathcal A}(
\widetilde{u}_j)\rightarrow a+\J,
\end{equation*}
for every $a\in\mathcal A$. Let $D:\mathcal A\longrightarrow X^*$ be a bounded module derivation, where $X$ is a $\mathcal A$-pseudo-unital (see \cite[Lemma 2.1]{am1}). Let $\|D\|\leq M$. Define $\Psi:{\mathcal A}\widehat
\otimes _{\mathfrak A} {\mathcal A}\longrightarrow X^*$ by $\Psi(a\otimes b)=D(a)\cdot b$. Thus, $\Psi$ is a bounded $\mathfrak A$-module map and $\|\Psi\|\leq\|D\|$. For every $a\in\mathcal A$ we have
\begin{equation}\label{e1}
    \Psi(\widetilde{u}_j \cdot a- a\cdot \widetilde{u}_j)=\Psi(\widetilde{u}_j)\cdot a-D(a)\cdot \widetilde{\omega}_{\mathcal A}(\widetilde{u}_j)-a\cdot \Psi(\widetilde{u}_j).
\end{equation}
Let $\Psi(\widetilde{u}_j)=-\xi_j$. Then (\ref{e1}) implies that
\begin{equation}\label{e2}
    D(a)\cdot \widetilde{\omega}_{\mathcal A}(\widetilde{u}_j)=a\cdot\xi_j-\xi_j\cdot a-\Psi(\widetilde{u}_j \cdot a- a\cdot \widetilde{u}_j).
\end{equation}
Since $X$ is $\mathcal A$-pseudo-unital, $D(a)\cdot \widetilde{\omega}_{\mathcal A}(\widetilde{u}_j)\stackrel{\emph{\emph{w}}^*}{\longrightarrow}D(a)$. On the other hand,
\begin{equation}\label{e3}
    \|\Psi(\widetilde{u}_j \cdot a- a\cdot \widetilde{u}_j)\|\leq M\|\widetilde{u}_j \cdot a- a\cdot \widetilde{u}_j\|\rightarrow0.
\end{equation}
as $j\rightarrow\infty$. It follows from \cite[Lemma 3.2]{bp}, (\ref{e2}) and (\ref{e3}) that $D(a)=\lim_j(a\cdot\xi_j-\xi_j\cdot a)$, for all $a\in\mathcal A$.

Conversely, assume that $\mathcal A$ is module approximately amenable with bounded approximate identity, say $(e_i)$. Corollary 3.6 of \cite{bp} shows that $\mathcal A$ is module approximately contractible as an $\mathfrak A^\sharp$-module. Applying Theorem 3.4 of \cite{bp}, $\mathcal B$ has a module approximate diagonal $\{m_j\}$ as an $\mathfrak A^\sharp$-module. We put $m_j$ as follows:
\begin{equation*}
    m_j=u_j+v_j\otimes _{\mathfrak A^\sharp} \textbf{1}+\textbf{1}\otimes _{\mathfrak A^\sharp} w_j+\textbf{1}\otimes _{\mathfrak A^\sharp}\textbf{1},
\end{equation*}
in which $u_j\in\mathcal A\otimes _{\mathfrak A^\sharp}\mathcal A$, $v_j,w_j\in\mathcal A$ and $\textbf{1}$ is unit of $\mathfrak A^\sharp$ (we can also construct these nets by $(e_i)$). Now consider the following net
\begin{equation*}
    n_{i,j}=u_j+v_j\otimes _{\mathfrak A^\sharp} e_i+e_i\otimes _{\mathfrak A^\sharp} w_j+e_i\otimes _{\mathfrak A^\sharp}e_i.
\end{equation*}
Let $I$ be the index set for the net $(e_i)$ and $J$ be the index set for the nets $(u_j)$, $(v_j)$, and $(w_j)$. We make the required net
$(n_k)$ by using an iterated limit construction (see \cite[page 69]{ke}). Our
indexing directed set is defined to be $K=I\times\Pi_{i\in
I}J$, equipped with the product ordering, and for each
$k=(i,f)\in K$, we define $n_k=m_{i,f(i)}$. Now, let
\begin{equation*}
    n_k=u_k+v_k\otimes _{\mathfrak A} e_k+e_k\otimes _{\mathfrak A} w_k+e_k\otimes _{\mathfrak A}e_k.
\end{equation*}
It is easy to check that $\{n_k\}$ is module diagonal for $\mathcal A$. This finishes the  proof.
\end{proof}

The corresponding result
characterizing pseudo-contractibility of a Banach algebra which is obtained in \cite[Theorem 2.4]{gz} is as follows.
\begin{theo}
Let $\mathcal A$  be a Banach ${\mathfrak A}$-module. Then the following statements are equivalent:
\begin{itemize}
  \item[(i)] $\mathcal A^\sharp$ is $\mathfrak{A}^\sharp$-module pseudo-contractible;
  \item[(ii)]$\mathcal A$ is $\mathfrak{A}^\sharp$-module pseudo-contractible and has identity;
  \item[(iii)] $\mathcal A$ is $\mathfrak{A}^\sharp$-module contractible.
\end{itemize}

In addition, if $\mathcal A$ is a left or right essential $\mathfrak{A}$-module, the above statements are equivalent with following statement:
\begin{itemize}
  \item[(iv)] $\mathcal A$ is $\mathfrak{A}$-module contractible.
\end{itemize}
\end{theo}
\begin{proof}
(i)$\Rightarrow$(ii)  Let $(\widetilde{U}_j)\subseteq\mathcal A^\sharp\widehat{\otimes}_{\mathfrak{A}^\sharp}\mathcal A^\sharp$ be a central approximate diagonal for $\mathcal A^\sharp$ such that $a\cdot \widetilde{U}_j=\widetilde{U}_j\cdot a$ and $\widetilde{\omega}_{\mathcal A^\sharp}(\widetilde{U}_j)\rightarrow1_\mathcal A$ for every $a\in \mathcal A$. We have
\begin{equation*}
    \mathcal A^\sharp\widehat{\otimes}_{\mathfrak{A}^\sharp}\mathcal A^\sharp=(\mathcal A\widehat{\otimes}_{\mathfrak{A}^\sharp}\mathcal A)\oplus
    (\mathcal A\widehat{\otimes}_{\mathfrak{A}^\sharp}\mathfrak{A}^\sharp)\oplus(\mathfrak{A}^\sharp\widehat{\otimes}_{\mathfrak{A}^\sharp}\mathcal A)
    \oplus(\mathfrak{A}^\sharp\widehat{\otimes}_{\mathfrak{A}^\sharp}\mathfrak{A}^\sharp),
\end{equation*}
and thus we can write
\begin{equation*}
    \widetilde{U}_j=\widetilde{u}_j+F_j{\otimes}_{\mathfrak{A}^\sharp}1_\mathcal A+1_\A{\otimes}_{\mathfrak{A}^\sharp}G_j
    +c_j1_\A{\otimes}_{\mathfrak{A}^\sharp}1_\mathcal A,
\end{equation*}
where $\widetilde{u}_j\in\mathcal A{\otimes}_{\mathfrak{A}^\sharp}\mathcal A$, $a_j, b_j\in\mathcal A$, and $\lambda_j\in\mathbb{C}$. Then for every $a\in\mathcal A$ we have
\begin{equation}\label{theo0}
    a\cdot \widetilde{u}_j+a{\otimes}_{\mathfrak{A}^\sharp}b_j= \widetilde{u}_j\cdot a+a_j{\otimes}_{\mathfrak{A}^\sharp}a,\hspace{0.5cm}\widetilde{\omega}_\mathcal A(\widetilde{u}_j)+a_j+b_j\rightarrow0,
\end{equation}
and
\begin{equation}\label{theo1}
    aa_j=-\lambda_ja,\hspace{0.5cm}b_ja=-\lambda_ja, \hspace{0.5cm}\lambda_j\rightarrow1.
\end{equation}
Since $\mathcal A^\sharp$ is $\mathfrak{A}^\sharp$-module pseudo-contractible, $\mathcal A$ is $\mathfrak{A}^\sharp$-module approximately contractible (\cite[Theorem 3.4]{bp}) and this implies that $\mathcal A$ is $\mathfrak{A}^\sharp$-module approximately amenable (\cite[Corollary 3.6]{bp}). It follows from Lemma \ref{l1} that $\mathcal A$ has left and right approximate identity, say $(e_i)$ and $(f_i)$, respectively. By (\ref{theo1}), we have $e_ia_j=-\lambda_je_i$ and $b_jf_i=-\lambda_jf_i$. Set $e_l=\lim_i e_i$ and $e_r=\lim_j f_i$. Then $e_l=e_r=e$ is an identity for $\mathcal A$. Therefore $a_j=-\lambda_je$ and $b_j=-\lambda_j e$. Using (\ref{theo0}), we get
\begin{equation*}
    a\cdot (\widetilde{u}_j-\lambda_je{\otimes}_{\mathfrak{A}^\sharp}e)= (\widetilde{u}_j -\lambda_je{\otimes}_{\mathfrak{A}^\sharp}e)\cdot a,\hspace{0.5cm}\widetilde{\omega}_\mathcal A(\widetilde{u}_j)-2\lambda_je\rightarrow 0.
\end{equation*}
This means that the net $(\widetilde{U}_j)_j$ defined as above is the desired net.

(ii)$\Rightarrow$(iii) Let $(\widetilde{u}_j)$ be module central
approximate diagonal for $\mathcal A$ and $e$ be its unit such that $\widetilde{\omega}_j(\widetilde{u}_j)\rightarrow e$. Similar to the proof of Theorem 2.4 ((ii)$\Rightarrow$(iii)) in \cite{gz}, we conclude that (iii) holds.

(iii)$\Rightarrow$(i) In light of \cite[Theorem 2.3.11]{bod}, $\mathcal A^\sharp$ is $\mathfrak{A}^\sharp$-module pseudo-contractible.\\
Now, let $\mathcal A$ be left or right essential $\mathfrak{A}$-module and (iv) hold. Since every $\mathfrak{A}^\sharp$-module homomorphism is also an $\mathfrak{A}$-module homomorphism, (iv) implies (iii).

Conversely, let (iii) hold, and let $X$ be a commutative Banach $\mathcal A$-$\mathfrak{A}$-module and $D:\mathcal A\longrightarrow X$ be a $\mathfrak{A}^\sharp$-module derivation. Since $X$ become a Banach $\mathcal A$-$\mathfrak{A}^\sharp$-module and $\mathcal A$ is left (right) essential $\mathfrak{A}$-module, $D$ is a bounded derivation (see \cite{bp}) and we have
 $$D((\alpha,\lambda)\cdot a)=(\alpha,\lambda)\cdot D(a),~\emph{\emph{and}}~~\ D(a\cdot (\alpha,\lambda))=D(a)\cdot(\alpha,\lambda),$$
 for every $a\in\mathcal A$, $\alpha\in\mathfrak{A}$ and $\lambda\in\mathbb{C}$. Hence $D$ is an $\mathfrak{A}^\sharp$-module derivation which is inner. Therefore $\mathcal A$ is $\mathfrak{A}$-module contractible.
\end{proof}

Consider the module projective tensor product ${\mathcal
A}\hat{\otimes} _{\mathfrak A^\sharp} {\mathcal A}$ which is isomorphic
to the quotient space $(\mathcal A \hat{\otimes} \mathcal
A)/{\widetilde{\mathcal I}_{\mathcal A}}$, where $\widetilde{\mathcal I}_{\mathcal A}$ is the closed linear span of $\{ a\cdot\alpha \otimes b-a
\otimes\alpha \cdot b : \alpha\in {\mathfrak A^\sharp},a,b\in{\mathcal A}\}$, and the closed ideal
$\widetilde{\mathcal J}_{\mathcal A}$ of ${\mathcal A}$ generated by elements of the form $
(a\cdot\alpha) b-a(\alpha \cdot b)$ for $ \alpha\in {\mathfrak
A^\sharp},a,b\in{\mathcal A}$.  Also,  let
$\overline{\omega}_{\mathcal{A}} : \mathcal{A}\hat{\otimes}_{\mathfrak{A}^\sharp}\mathcal{A} =
{(\mathcal{A}\hat{\otimes}\mathcal{A}})/{\widetilde{{\mathcal I}}_\mathcal A}\longrightarrow
{\mathcal{A}}/{\widetilde{{\mathcal J}}_\mathcal A}$ defined by
$\overline{\omega}_{\mathcal{A}}(a\otimes
b+\widetilde{{\mathcal I}}_\mathcal A)=ab+\widetilde{{\mathcal J}}_\mathcal A$.
\begin{prop}
Let $\mathcal A$ be a commutative Banach $\mathfrak{A}$-module with central approximate identity. If $\mathcal A$ is $\mathfrak{A}^\sharp$-module approximately amenable, then it is $\mathfrak{A}^\sharp$-module pseudo-amenable. Furthermore, if $\A$ is left or right essential $\mathfrak{A}$-module, then $\mathfrak{A}$-module approximate amenability of $\mathcal A$ implies its $\mathfrak{A}$-module pseudo-amenability.
\end{prop}
\begin{proof}
Let $(e_i)$ be the central approximate identity for $\mathcal A$. Given $\varepsilon>0$ and finite set $F\subset\mathcal A$, choose $e_{i_1},e_{i_2}\in(e_i)$ such that for every $a\in F$
\begin{equation*}
    \|e_{i_1}a-a\|<\varepsilon/2\quad \text{and}\quad\|e_{i_2}e_{i_1}a-e_{i_1}a\|<\varepsilon/2.
\end{equation*}
Consider $X=\ker\overline{\omega}_\mathcal A$ as a commutative $\mathcal A$-$\mathfrak{A}^\sharp$-module. Define $D:\mathcal A\longrightarrow X$ by $$D(a)=ae_{i_1}\otimes_{\mathfrak{A}^\sharp}e_{i_2}-e_{i_1}\otimes_{\mathfrak{A}^\sharp}e_{i_2}a,$$
for all $a\in\mathcal A$. Clearly, $D$ is a $\mathfrak{A}^\sharp$-module derivation. Our assumption implies that $\mathcal A$ is  $\mathfrak{A}^\sharp$-module approximately contractible (\cite[Corollary 3.6]{bp}). Thus, there exists $x=x(e_{i_1},e_{i_2},\varepsilon, F)\in X$ such that $\|D(a)-(a\cdot u-u\cdot a)\|<\varepsilon$, for all $a\in F$. Put $U=e_{i_1}\otimes_{\mathfrak{A}^\sharp} e_{i_2}-u$. Obviously, $U\in\mathcal A\otimes_{\mathfrak{A}^\sharp}\mathcal A$ and $$a\cdot U-U\cdot a\rightarrow0, \hspace{1cm}\overline{\omega}(U)a-a\rightarrow0,$$
for all $a\in\mathcal A$. This means that $(U)$ is a module central approximate diagonal for $\mathcal A$. The rest of the proof is clear by repeating the above statements and using Corollary 3.6 in \cite{bp}.
\end{proof}

In general case, $\mathcal J\cdot(\mathcal A \widehat \otimes \mathcal A)$ is not a subset of $\mathcal I$ and thus $(\mathcal A \widehat \otimes \mathcal A)/{\mathcal I}$ is not always an $\mathcal A/\mathcal J$-module. We say the Banach algebra ${\mathfrak A}$ acts trivially on
$\mathcal A$ from left (right) if there is a continuous linear
functional $f$ on ${\mathfrak A}$ such that  $\alpha\cdot
a=f(\alpha)a$ ($a\cdot\alpha=f(\alpha)a$), for each $\alpha\in
\mathfrak A$ and $a\in \mathcal A$ (see also \cite {abe}). The following lemma is proved in \cite[Lemma 3.13]{bab}.

\begin{lemm} \label{tpc2} {\it If ${\mathfrak
A}$ acts on $\mathcal A$ trivially from the left or right and
$\mathcal A/\mathcal J$ has a right bounded approximate identity, then for
each $\alpha \in {\mathfrak A}$ and $a\in\mathcal A$ we have
$f(\alpha)a-a\cdot\alpha \in \mathcal J$.}
\end{lemm}

We note that if all conditions of Lemma \ref{tpc2} hold, then $\mathcal A/\mathcal J$ is a commutative ${\mathfrak A}$-module. We now show that when $\mathcal A/\mathcal J$ has a right bounded approximate identity,
then $(\mathcal A \widehat \otimes \mathcal A)/{\mathcal I}$  is $\mathcal A/\mathcal J$-module if ${\mathfrak A}$ acts on $\mathcal A$ trivially
from left or right. For the case of the trivial left action,
consider the following actions
$$(a+\mathcal J)\cdot(b\otimes c+\mathcal I)=ab\otimes c+\mathcal I,\quad (b\otimes c+\mathcal I)\cdot (a+\mathcal J)=b\otimes ca+\mathcal I.$$
For $a,b,c\in \mathcal A$ and $\alpha\in \mathfrak A$, we have
$$[a\cdot\alpha-f(\alpha)a]\cdot(b\otimes c)=(a\cdot\alpha)b\otimes c-f(\alpha)ab\otimes
c=(a\cdot\alpha)b\otimes c-ab\otimes \alpha\cdot c\in \mathcal I.$$ Thus
left action is well-defined. Similarly, one can show that the
right action is also well-defined.

\begin{prop}\label{pppp} Let ${\mathfrak A}$ act trivially on
$\mathcal A$ from left \emph{(}or right\emph{)} such that each approximate identity of $\mathcal A/\mathcal{J}$  is also an approximate
identity for $X$. If $\mathcal A$ is $\mathfrak{A}$-module pseudo-amenable and $X$ be a commutative Banach $\mathcal A$-$\mathfrak{A}$-module, then
\begin{itemize}
  \item[(i)] every continuous module derivation from $\mathcal A/\mathcal{J}$ into $X$, is approximately inner.
  \item[(ii)] every continuous module derivation from $\mathcal A/\mathcal{J}$ into $X^*$, is $w^*$-approximately inner.
\end{itemize}
\end{prop}
\begin{proof}
 Let $(\widetilde{u}_j)\subseteq\mathcal A\widehat{\otimes}_\mathfrak{A}\mathcal A$ be a module approximate diagonal for $\mathcal A$ and $(\widetilde{{\omega}}_\mathcal A(\widetilde{u}_j))$ be a right and left approximate identity for $X$.

 (i) Let $D:\mathcal A/\mathcal{J}\longrightarrow X$ be a continuous module derivation. Assume $\widetilde{u}_j=\sum_ia_i^{(j)}\otimes_\mathfrak{A}b_i^{(j)}=\sum_i(a_i^{(j)}\otimes b_i^{(j)}+\mathcal{I})$ and let $\xi_j=\sum_iD(a_i^{(j)}b_i^{(j)}+\mathcal{J})$. Then
\begin{eqnarray}\label{12}
 \nonumber
  D((a+\mathcal{J})(a_{i}^{(j)} b_{i}^{(j)}+\mathcal{J})) &=& D(a+\mathcal{J})\cdot(a_{i}^{(j)} b_{i}^{(j)}+\mathcal{J})+(a+\mathcal{J})\cdot D(a_{i}^{(j)} b_{i}^{(j)}+\mathcal{J}) \\
  \nonumber
   &=&  D(a+\mathcal{J})\cdot(a_{i}^{(j)} b_{i}^{(j)}+\mathcal{J})+(a+\mathcal{J})\cdot D(a_{i}^{(j)}+\mathcal{J})\cdot(b_{i}^{(j)}+\mathcal{J})\\
   &&  +(aa_{i}^{(j)}+\mathcal{J})\cdot D(b_{i}^{(j)}+\mathcal{J}),
\end{eqnarray}
and
\begin{eqnarray}\label{123}
 \nonumber
  D((a_{i}^{(j)} b_{i}^{(j)}+\mathcal{J})(a+\mathcal{J})) &=& D(a_{i}^{(j)} b_{i}^{(j)}+\mathcal{J})\cdot (a+\mathcal{J})+(a_{i}^{(j)} b_{i}^{(j)}+\mathcal{J})\cdot D(a+\mathcal{J}) \\
  \nonumber
   &=&  D(a_{i}^{(j)}+\mathcal{J}) \cdot(b_{i}^{(j)}a+\mathcal{J})+ (a_{i}^{(j)}+\mathcal{J})\cdot D(b_{i}^{(j)}+\mathcal{J})\cdot(a+\mathcal{J})\\
   &&  +(a_{i}^{(j)} b_{i}^{(j)}+\mathcal{J})\cdot D(a+\mathcal{J}).
\end{eqnarray}
For every $a\in\mathcal A$, we have

$$(a+\mathcal{J})\widetilde{\omega}_\mathcal A(\sum_i(a_i^{(j)}\otimes_\mathfrak{A}b_i^{(j)}+\mathcal{I}))-\widetilde{\omega}_\mathcal A(\sum_i(a_i^{(j)}
\otimes_\mathfrak{A}b_i^{(j)}+\mathcal{I}))(a+\mathcal{J})
\rightarrow0.$$
Due to the continuity of $D$
\begin{equation}\label{22}
    D((a+\mathcal{J})\widetilde{\omega}_\mathcal A(\sum_i(a_i^{(j)}\otimes_\mathfrak{A}b_i^{(j)}+\mathcal{I}))-\widetilde{\omega}_\mathcal A(\sum_i(a_i^{(j)}
\otimes_\mathfrak{A}b_i^{(j)}+\mathcal{I}))(a+\mathcal{J}))
\rightarrow0,
\end{equation}
for all $a\in\mathcal A$. Since $\widetilde{\omega}_\A(\widetilde{u}_j)$ is left approximate identity for $X$ and approximate identity for $\mathcal A/\mathcal{J}$, by (\ref{123}), we get
\begin{equation}\label{14}
    D(a_{i}^{(j)}+\mathcal{J})\cdot(b_{i}^{(j)}a+\mathcal{J}) +(a_{i}^{(j)}+\mathcal{J})\cdot D(b_{i}^{(j)}+\mathcal{J})\cdot(a+\mathcal{J})\rightarrow0.
\end{equation}
From (\ref{12}), (\ref{22}) and (\ref{14}) we deduce that
\begin{equation}\label{}
  (a+\mathcal{J})\cdot\xi_j-\xi_j\cdot(a+\mathcal{J})+D(a+\mathcal{J})\cdot\widetilde{{\omega}}_\mathcal A(\widetilde{u}_j)\rightarrow0.
\end{equation}
Again by our assumption that $(\widetilde{{\omega}}_\mathcal A(\widetilde{u}_j))$ is right approximate identity for $X$, we conclude that  $D(a+\mathcal{J})\widetilde{{\omega}}_\mathcal A(\widetilde{u}_j)\longrightarrow D(a+\mathcal{J})$, for all $a\in\mathcal A$. This implies
$$ D(a+\mathcal{J})=\lim_j \xi_j\cdot(a+\mathcal{J})-(a+\mathcal{J})\cdot\xi_j, (a\in\mathcal A).$$

(ii) Let $D:\mathcal A\longrightarrow X^*$ be a module derivation. Setting $\xi_j=\sum_iD(a_i^{(i)}\otimes b_i^{(i)}+\mathcal{J})$, we have
$$(a+\mathcal{J})\cdot\xi_j-\xi_j\cdot(a+\mathcal{J})+D(a+\mathcal{J})\widetilde{{\omega}}_\mathcal A(\widetilde{u}_j)\rightarrow0.$$
It follows from the property of $(\widetilde{\omega}(\widetilde{u}_j))$ that $\widetilde{\omega}_{\mathcal A}(\widetilde{u}_j)\cdot D(a)\stackrel{w^*}{\longrightarrow}D(a)$. Then
$$ D(a+\mathcal{J})=\lim_j \xi_j\cdot(a+\mathcal{J})-(a+\mathcal{J})\cdot\xi_j,$$
for all $a\in\mathcal A$.
\end{proof}

The next consequence was proved in \cite{gz}, but the net $(\xi_\alpha)$ which is defined in its proof does not satisfy in the condition $a\xi_\alpha-\xi_\alpha a- D(a)\pi(u_\alpha)\rightarrow0$. We point out that this result can be deduced as a corollary of Proposition \ref{pppp}.

\begin{cor}
Let $\mathcal A$ be a pseudo-amenable Banach algebra, and let $X$ be a Banach $\mathcal A$-bimodule
such that each approximate identity of $A$ is also an approximate
identity for X. Then:
\begin{itemize}
  \item[(i)] every continuous module derivation from $\mathcal A$ into $X$, is approximately inner.
  \item[(ii)] every continuous module derivation from $\mathcal A$ into $X^*$, is $w^*$-approximately inner.
\end{itemize}
\end{cor}

\begin{prop}\label{p}
Let $\mathcal A$ and $\mathcal B$ be Banach ${\mathfrak
A}$-modules. If there is a continuous ${\mathfrak A}$-module
epimorphism from $\mathcal A$ onto $\mathcal B$ and $\mathcal A$
is module pseudo-amenable, so is $\mathcal B$. In particular, the
quotient algebra $\mathcal A/I$ is module pseudo-amenable for any
two-sided closed ideal $I$ of $\mathcal A$.
\end{prop}
\begin{proof}
Assume that $\phi:\mathcal A\longrightarrow \mathcal B$ is a continuous ${\mathfrak A}$-module
epimorphism. Then the
map $\overline{\phi\otimes\phi}:{\mathcal A}\widehat \otimes
_{\mathfrak A} {\mathcal A}\cong({\mathcal A}\widehat
\otimes{\mathcal A})/\mathcal I_{\mathcal A}\longrightarrow
{\mathcal B}\widehat \otimes _{\mathfrak A} {\mathcal
B}\cong({\mathcal B}\widehat \otimes{\mathcal B})/\mathcal
I_{\mathcal B}$ defined by $\overline{\phi\otimes\phi}(a\otimes
b+\mathcal I_{\mathcal A})=\phi(a)\otimes\phi(b)+\mathcal
I_{\mathcal B}$ takes any module approximate diagonal for
$\mathcal A$ to a module approximate diagonal for $\mathcal B$.
\end{proof}

Recall that the  convex hull of a subset $A$ of a normed space $X$, denoted by ${\rm{co}}(A)$, is the intersection of all convex sets in $X$ that contains $A$.
\begin{theo} \label{t4}
Let $\mathcal A$ be a Banach $\mathfrak A$-bimodule with compatible actions. Then the following statements are equivalent:
\begin{enumerate}
\item[$\text{(i)}$] {$\mathcal A$ is module approximately amenable;}
\item[$\text{(ii)}$] {For any commutative $\mathcal A$-$\mathfrak A$-module $X$, every bounded derivation $D:\mathcal A\longrightarrow X^{**}$ is approximately inner.}
\end{enumerate}
\end{theo}
\begin{proof}
(i)$\Rightarrow$(ii)  This is trivial.

(ii)$\Rightarrow$(i)  It is sufficient to show that every module derivation $D:\mathcal A\longrightarrow X$ is approximately inner. If $i:X\longrightarrow X^{**}$ is the canonical embedding, then $\tilde{D}=i\circ D$ is a module derivation from $\mathcal A$ into $X^{**}$. By assumption, there exists a net $(\Phi_j)$ in $X^{**}$ such that
$$\tilde{D}(a)=\lim_j(a\cdot\Phi_j-\Phi_j\cdot a),\quad(a\in \mathcal A).$$
Take $\epsilon>0$ and finite sets $\mathcal F\subset \mathcal A$, $E\subset X^*$. Then there is a $j$ such that
$$|\langle \tilde{D}(a)-(a\cdot\Phi_j-\Phi_j\cdot a), f\rangle|<\epsilon \qquad (f\in E, \, a\in \mathcal F).$$
By Goldstien's Theorem, there is a $x_j\in X$ such that
$$|\langle f, \tilde{D}(a)-(a\cdot x_j-x_j\cdot a)\rangle|<\epsilon \qquad (f\in E, \, a\in \mathcal F).$$
Therefore there is a net $(x_j)j\in \Gamma\subset X$ so that $D(a)= \lim_j(a\cdot x_j-x_j\cdot a)$ weakly in $X$. Now for each finite set $\mathcal F\subset \mathcal A$, say $\mathcal F=\{a_1, \ldots, a_k\}$,
$$(a_1\cdot x_j -x_j\cdot a_1, \ldots, a_k\cdot x_j -x_j \cdot a_k) \to (D(a_1), \ldots, D(a_k)),$$
weakly in $X^n$. Thus
$(D(a_1), \ldots, D(a_k))$ belongs to the weak closure of $$V={\rm{co}}\{(a_1\cdot x_j -x_j\cdot a_1, \ldots, a_k\cdot x_j -x_j \cdot a_k) : j\in \Gamma\}.$$
By Mazur's Theorem, $(D(a_1), \ldots, D(a_k))$ belongs to the norm closure of $V$. Hence, for each $\epsilon >0$, there is $u_{F, \epsilon}\in {\rm{co}}\{x_j\}$ such that
$$\|D(a)-(a\cdot u_{F, \epsilon}-u_{F, \epsilon}\cdot a)\|<\epsilon, \qquad (a\in \mathcal F).$$
The last inequality shows that $\mathcal A$ is module approximately amenable.
\end{proof}

We denote by $\square$ the first Arens product on
$\mathcal{A}^{**}$, the second dual of $\mathcal{A}$. Here and subsequently, $\mathcal A^{**}$ is equipped with the
first Arens product.

\begin{theo}\label{thmt}
Let $\mathcal A$ be a Banach $\mathfrak A$-module such that
$\mathcal A\hat{\otimes}_{\mathfrak A}\mathcal A$ is a commutative Banach
$\mathcal A$-$\mathfrak A$-module. Then module pseudo-amenability of $\mathcal A^{**}$ implies module pseudo-amenability of $\mathcal A$.
\end{theo}
\begin{proof} According to the argument of the proof of \cite[Proposition 3.7]{bp}, we see that under our assumption there is a net $(\widetilde{m}_i)\subseteq
({\mathcal A}\widehat \otimes _{\mathfrak A} {\mathcal
A})^{**}$ such that $\tilde{\omega}_{\mathcal A}^{**}(\widetilde{m}_i)\cdot
(a+\mathcal J) \rightarrow a+\mathcal J$ and $a\cdot\widetilde{m}_i- \widetilde{m}_i\cdot a
\rightarrow 0$ for all $a\in\mathcal A$. Now, we can use Goldstien's Theorem to
obtain $(\widetilde{m}_i)$ in ${\mathcal A}\widehat \otimes _{\mathfrak A} {\mathcal
A}$, and we can replace weak$^*$ convergence in the above two limits by weak convergence. This implies, by Mazur's Theorem, that ${\mathcal A}$ is module
pseudo-amenable.
\end{proof}

Let $X$ and $Y$ be Banach spaces. Then the weak$^*$ operator topology on $B(X,Y^*)$
is the locally convex topology determined by the seminorms $\{p_{x,y}: x\in X, y\in Y\}$, where $p_{x,y}(\Phi) = \left\langle \Phi(x),y\right\rangle$.

Let $X$ and $Y$ be Banach ${\mathcal A}$-modules and Banach ${\mathfrak A}$-modules. A net $(T_k)$ of bounded maps from $X$ into $Y$, satisfying
\begin{equation}\label{a11}\|T_k (a \cdot x)-a \cdot T_k(x)\|\rightarrow 0,\quad \|T_k (x\cdot a)-T_k(x)\cdot a\|\rightarrow 0,\quad(a\in\mathcal A, x\in X),\end{equation}
\begin{equation}\label{a12}\|T_k (\alpha \cdot x)-\alpha \cdot T_k(x)\|\rightarrow 0,\quad \|T_k (x\cdot \alpha)-T_k(x)\cdot \alpha\|\rightarrow 0,\quad(\alpha\in\mathfrak A, x\in X),\end{equation}
is said to be an {\it module approximate morphism} from
$X$ to $Y$. If $Y$ is a dual Banach space, and we replace norm convergence
by $w^*$-convergence in (\ref{a11}) and (\ref{a12}) , we call $(T_k)$ a module $w^*$-approximate morphism. The following theorem is analogous to \cite[Theorem 2.4]{sam} in the case of module pseudo-amenability. We include the proof.

\begin{theo} \label{tt} Let $\mathcal A\widehat \otimes _{\mathfrak A} \mathcal A$ be commutative as an $\mathcal A$-$\mathfrak A$-module. Consider the following conditions:
\begin{enumerate}
\item[$\text{(i)}$] {$\mathcal A$ is module pseudo-amenable;}
\item[$\text{(ii)}$] {there is a module approximate morphism $(S_l)$ from $\mathcal A/\mathcal J$ into $\mathcal A\widehat \otimes _{\mathfrak A} \mathcal A$
such that
$$\|\tilde{\omega}_{\mathcal{A}}\circ S_l(a+\mathcal J)-(a+\mathcal J)\|\rightarrow 0\qquad(a\in \mathcal A);$$}
\item[$\text{(iii)}$] {there is a module $w^*$-approximate morphism $T_l:(\mathcal A\widehat \otimes _{\mathfrak A} \mathcal A)^*\longrightarrow(\mathcal A/\mathcal J)^*$ such that $\lim_l T_l\circ\tilde{\omega}_{\mathcal{A}}^*=id_{(\mathcal A/\mathcal J)^*}$ in weak$^*$ operator topology.}
\end{enumerate}

Then $\emph{(i)}\Rightarrow\emph{(ii)}\Rightarrow\emph{(iii)}$. In addition, if $\mathcal A/\mathcal J$ has a central approximate identity, then all assertions are equivalent.
\end{theo}
\begin{proof} $\text{(i)}\Rightarrow\text{(ii)}$ Suppose that $\{\widetilde{u}_j \}$ is a module
approximate diagonal for $\mathcal A$. Since $\mathcal A\widehat \otimes _{\mathfrak A} \mathcal A$ is commutative $\mathfrak A$-module, the map $S_l: \mathcal A/\mathcal J\longrightarrow\mathcal A\widehat \otimes _{\mathfrak A} \mathcal A;\,a+\mathcal J\mapsto \widetilde{u}_j\cdot(a+\mathcal J)$ satisfies the properties of condition (ii).

 $\text{(ii)}\Rightarrow\text{(iii)}$ The map $T_l$, the dual of $S_l$ satisfies the conditions of assertions (iii).

$\text{(iii)}\Rightarrow\text{(i)}$ Assume that $(v_k)_{k\in \Gamma}$ is a central approximate identity for $\mathcal A/\mathcal J$ and $T_l:(\mathcal A\widehat \otimes _{\mathfrak A} \mathcal A)^*\longrightarrow(\mathcal A/\mathcal J)^*\,\,(l\in\Sigma)$ satisfies the conditions of statement (iii). For each $a\in \mathcal A$ and $\varphi\in(\mathcal A\widehat \otimes _{\mathfrak A} \mathcal A)^*$, we have
\begin{align*}
\lim_k\lim_l\left\langle a\cdot T_l^*(v_k)-T_l^*(v_k)\cdot a, \varphi\right\rangle&=\lim_k\lim_l\left\langle  T_l^*(v_k),\varphi\cdot a-a\cdot\varphi\right\rangle \\
&=\lim_k\lim_l\left\langle T_l(\varphi\cdot a)-T_l(a\cdot\varphi),v_k \right\rangle\\
&=\lim_k\lim_l\left\langle  T_l(\varphi\cdot a)-T_l(\varphi)\cdot a+T_l(\varphi)\cdot a-T_l(a\cdot\varphi),v_k\right\rangle\\
&=\lim_k\lim_l\left\langle  T_l(\varphi\cdot a)-T_l(\varphi)\cdot a,v_k\right\rangle\\
&+\lim_k\lim_l\left\langle a\cdot T_l(\varphi)-T_l(a\cdot\varphi),v_k\right\rangle\\
&=\lim_k(0+0)=0.
\end{align*}
Note that module action $\mathcal A$ over $\mathcal A/\mathcal J$ and the product in $\mathcal A/\mathcal J$ coincide. Also, for each $a\in \mathcal A$ and $f\in(\mathcal A/\mathcal J)^*$, we get
\begin{align*}
\lim_k\lim_l\left\langle \tilde{\omega}_{\mathcal{A}}^{**}(T^*_l(v_k))\cdot (a+\mathcal J), f\right\rangle&=\lim_k\lim_l\left\langle  T_l^*(v_k),\tilde{\omega}_{\mathcal{A}}^{*}((a+\mathcal J)\cdot f)\right\rangle \\
&=\lim_k\lim_l\left\langle  T_l(\tilde{\omega}_{\mathcal{A}}^{*}((a+\mathcal J)\cdot f)),v_k\right\rangle\\
&=\lim_k\left\langle  (a+\mathcal J)\cdot f,v_k\right\rangle\\
&=\lim_k\left\langle f,v_k\cdot(a+\mathcal J)\right\rangle\\
&=\left\langle f,a+\mathcal J\right\rangle.
\end{align*}
Let $\mathfrak E =\Gamma\times\Sigma^{\Gamma}$ be directed by the product ordering and for each $p=(k,(l_{k'}))\in\mathfrak E$, let $n_p=T_{l_k}(v_k)\in (\mathcal A\widehat \otimes _{\mathfrak A} \mathcal A)^{**}$. Using the iterated limit theorem
\cite[page 69]{ke} and the above calculations, we have $a\cdot n_p-n_p\cdot a\rightarrow 0$ weak$^*$ in $(\mathcal A\widehat \otimes _{\mathfrak A} \mathcal A)^{**}$ and $\tilde{\omega}_{\mathcal{A}}^{**}(n_p)\cdot(a+\mathcal J)\rightarrow a+\mathcal J$ weak$^*$ in $(\mathcal A/\mathcal J)^{**}$ for all $a\in \mathcal A$. Now, the same argument in the proof of Theorem \ref{thmt} can be applied to show that $\mathcal A$ is module pseudo-amenable.
\end{proof}

Recall that an {\it inverse semigroup} is a semigroup $S$ such that for each $s\in S$ there is a unique element $s^*\in S$ with $ss^*s=s$ and $s^*ss^*=s^*$. Elements of the form $s^*s$ are called {\it idempotents} and the set of all idempotents is denoted by $E_S$ (or $E$).

Let $S$ be a (discrete) inverse semigroup with with subsemigroup $E$ of idempotents. We say that $S$ is a {\it semilattice} if $S$ is commutative and $E_S=S$. It is easy to see that $E$  a semilattice  \cite[Theorem 5.1.1]{how} with the following order
$$e\leq d\Longleftrightarrow ed=e\qquad (e,d \in E).$$
In particular $ \ell ^{1}(E)$ could be regarded as a subalgebra
of $ \ell ^{1}(S)$. Consequently, $ \ell ^{1}(S)$ is a Banach algebra and a
Banach $ \ell ^{1}(E)$-module with compatible actions \cite{am1}. Consider the actions of $ \ell ^{1}(E)$ on $ \ell
^{1}(S)$ as follows:
$$\delta_e\cdot\delta_s = \delta_s, \ \delta_s\cdot\delta_e = \delta_{se} =
\delta_s * \delta_e \qquad (s \in S,  e \in E).$$

From now on, we assume that $\ell ^1(S)$ is a Banach bimodule over $\ell ^1(E)$ with
the above actions. Let the ideal $\mathcal J$ be the closed linear
span of $\{\delta_{set}-\delta_{st}: \, s,t \in S,  e \in E
\}.$ We consider an equivalence relation on $S$ such that $s\approx t$ if and only if $\delta_s-\delta_t \in \mathcal J$ for all $s,t \in
S.$  It is shown in \cite{pou} that the quotient ${S}/{\approx}$ is a
discrete group (see also \cite{abe}). Indeed,
${S}/{\approx}$ is homomorphic to the maximal group homomorphic
image $G_S$ \cite{mn} of $S$ \cite{pou2}. Moreover, $S$ is
amenable if and only if $G_S={S}/{\approx}$ is amenable
\cite{dun,mn}.

On should remember that a Banach algebra $\mathcal A$ is pseudo-amenable if there is a net $(u_j)\subset{\mathcal A}\widehat
\otimes {\mathcal A}$,
called an approximate diagonal for $\mathcal A$, such that $a\cdot u_j-u_j\cdot a
\stackrel{j}{\longrightarrow}0$ and $\omega_{\mathcal A}(u_j)a \stackrel{j}{\longrightarrow}a$ for all
$a\in \mathcal A$ \cite{gz}.

To achieve our aim, we need the following lemma which is analogous to \cite[Theorem 2.2.15]{bod} for the module pseudo-amenability case. Since the proof is similar, it is omitted.

\begin{lemm}\label{lemma} Let $\mathcal A$ be a Banach $\mathfrak A$-module with trivial left action and $\mathcal A/\mathcal J$ has a right bounded approximate identity. Then $\mathcal A/\mathcal J$ is module pseudo-amenable if and only if it is pseudo-amenable.

\end{lemm}
\begin{theo}\label{tm} Let $S$ be a discrete inverse semigroup. Then

\begin{enumerate}
\item[$\text{(i)}$] {$\ell ^1(S)$ is module pseudo-amenable if and only if $S$ is amenable;}
\item[$\text{(ii)}$] {$\ell ^1(S)^{**}$ is module pseudo-amenable if and only if $G_S$ is finite.}
\end{enumerate}
\end{theo}
\begin{proof} (i)
First note that $\ell ^1(G)/\mathcal J\cong {\ell ^{1}}(G_S)$ (see \cite{ra1}) is a commutative Banach $\ell ^{1}(E)$-bimodule. If $\ell ^1(S)$ is module pseudo-amenable, then so is $\ell ^1(G_S)$
by Proposition \ref{p}. Since the group algebra $\ell ^1(G_S)$ has an identity, the module pseudo-amenability of $\ell ^1(G_S)$ is equivalent to its pseudo-amenability by Lemma \ref{lemma}. Now, it follows from \cite[Proposition 4.1]{gz} that  $G_S$ is amenable. Therefore $S$ is amenable by \cite[Theorem 1]{dun}. The converse is clear by \cite[Theorem 3.1]{am1}.

(ii)  If $G_S$ is finite, then $\ell ^1(S)^{**}$ module amenable by \cite[Theorem 3.4]{abe} (see also \cite[Theorem 2.11]{pou}). Thus $\ell ^1(S)^{**}$ is module pseudo-amenable.

Conversely, let $\mathcal N$ be the closed ideal of ${\mathcal
A^{**}}$ generated by $(F\cdot \delta_e )\square G-F \square
(\delta_e \cdot G)$, for $F,G\in\ell ^1(S)^{**}$ and $e\in
E$. Similar to the proof of \cite[Theorem 3.4]{abe} we can show that $\mathcal J\subseteq \mathcal N \subseteq \mathcal J^{\perp \perp}$. If $\ell ^1(S)^{**}$ is module pseudo-amenable, so is $\ell ^1(S)^{**}/\mathcal N$ by Proposition \ref{p}. Going back to the case where $\mathcal
A=\ell^1(S)^{**}$, $\mathcal J=\mathcal N$ and $\mathfrak A=\ell^1(E)$ in Lemma \ref{lemma}. Then $\ell ^1(S)^{**}/\mathcal N$ is pseudo-amenable. The map $\theta : \ell ^1(S)^{**}/\mathcal N\longrightarrow \ell ^1(S)^{**}/\mathcal J^{\perp \perp}$ defined by  $\theta(F+N)= F+ J^{\perp \perp}$ is a well defined continuous
epimorphism. Since homomorphic image of a pseudo-amenable Banach algebra under a continuous epimorphism is again pseudo-amenable \cite[Proposition 2.2]{gz}, $\ell ^{1}(S)^{**}/\mathcal J^{\perp \perp}\cong \ell^1(G_S)^{**}$ is pseudo-amenable. Now, Proposition 4.2 from \cite{gz} shows that $G_S$ is finite.
\end{proof}

 We finish the paper by two examples. First we give an example
for which $\mathcal A$ is module pseudo-amenable but not pseudo-amenable. This example indicates that this new notion of amenability is different from the classical case. The second example shows that  the semigroup algebra on the Brandt
inverse semigroup is module pseudo-amenable (contractible) and pseudo-amenable but not pseudo-contractible.

\begin{ex} Let $\mathcal C$ be the
bicyclic inverse semigroup generated by $p$ and $q$, that is
$$\mathcal C=\{p^mq^n : m,n\geq 0 \},\hspace{0.2cm}(p^mq^n)^*=p^nq^m. $$ 
One can easily check that $ss^*s=s$ and $s^*ss^*=s^*$. The set of idempotents of $\mathcal C$ is $E_{\mathcal
C}=\{p^nq^n : n=0,1,...\}$ which is totally ordered with the
following order
$$p^nq^n \leq p^mq^m \Longleftrightarrow m \leq n.$$
It is shown in \cite{abe} that $\mathcal C/\approx$ is isomorphic
to the group of integers $\mathbb{Z}$, hence $\mathcal C$ is
amenable (see also \cite{dun} and \cite{pou2}). Therefore, by Theorem \ref{tm}, ${\ell ^{1}}(\mathcal C)$ is module pseudo-amenable as an ${\ell
^{1}}(E_{\mathcal C})$-module but ${\ell ^{1}}(\mathcal C)^{**}$
is not. On the other hand, any Banach algebra with a bounded approximate identity is approximate
amenable if and only if it is pseudo-amenable \cite[Proposition 3.2]{gz}. Since ${\ell ^{1}}(\mathcal C)$ is not approximately amenable \cite{gy}, it is not  pseudo-amenable. 
\end{ex}

\begin{ex}
Let $G$ be a group with identity $e$, and let $I$ be a non-empty
set. Then the Brandt inverse semigroup corresponding to $G$ and
$I$, denoted by $S={\mathcal{M}}(G,I)$, is the collection of all
$I \times I$ matrices $(g)_{ij}$ with $g\in G$ in the
$(i,j)^{\textmd{th}}$ place and $0$ (zero) elsewhere and the $I
\times I$ zero matrix $0$. Multiplication in $S$ is given by the
formula
$$(g)_{ij}(h)_{kl}=\left\{\begin{array}{l} (gh)_{il} \qquad {\rm {if}}\
j=k \\ \ \ 0 \qquad \quad \, {\rm{if}}\ j \neq k
\end{array}\right. \qquad\qquad (g,h\in G, \, i,j,k,l\in I),$$ and
${(g)^*_{ij}}=(g^{-1})_{ji}$ and $0^*=0$. The set of all idempotents
is $E_S=\{(e)_{ii}:i\in I\}\bigcup \{0\}$.  It is shown in
\cite[Example 3.2]{pou} that ${\ell ^{1}}(S)$ is module contractible. But if the index set $I$ is
infinite, then ${\ell ^{1}}(S)$ is not pseudo-contractible \cite[Corollary 2.5]{erm}. It is well-known that in the case where $S={\mathcal{M}}(G,I)$, we have $G=G_S$. Since pseudo-amenability of ${\ell ^{1}}(S)$ is equivalent to the amenability of $G$ \cite{erp} and $G_S$ is the trivial group \cite{pou}, two concepts of  module pseudo-amenability and  pseudo-amenability on ${\ell ^{1}}(S)$ coincide.  \end{ex}

\section*{Acknowledgement}
The authors sincerely thank the anonymous reviewers for their
careful reading, constructive comments and fruitful suggestions
to improve the quality of the first draft. 


\end{document}